\documentclass[12pt,twoside]{amsart}
\usepackage{amsmath, amsthm, amscd, amsfonts, amssymb, graphicx}
\usepackage{enumerate}
\usepackage[colorlinks=true,
linkcolor=blue,
urlcolor=cyan,
citecolor=red]{hyperref}
\usepackage{mathrsfs}
\newtheorem{theorem}{Theorem}[section]
\newtheorem{lemma}{Lemma}[section]
\newtheorem{remark}{Remark}[section]

\newtheorem{corollary}{Corollary}[section]

\newtheorem{proposition}{Proposition}[section]
\numberwithin{equation}{section}

\begin{document}
	
\title{Some operator inequalities via convexity}
\author{Hamid Reza Moradi, Shigeru Furuichi and Mohammad Sababheh}

\begin{abstract}
In this article, we employ a standard convex argument to obtain new and refined inequalities related to the matrix mean of two accretive matrices, the numerical radius and the Tsallis relative operator entropy.
\end{abstract}
\maketitle
\pagestyle{myheadings}
\markboth{\centerline {Some operator inequalities via convexity}}
{\centerline {}}
\bigskip
\bigskip
\section{Introduction}
Consider a complex Hilbert space $\left( \mathcal H,\left\langle \cdot,\cdot \right\rangle  \right)$. Let $\mathcal B\left( \mathcal H \right)$ denote the algebra of all bounded linear operators acting on $\left( \mathcal H,\left\langle \cdot,\cdot \right\rangle  \right)$
 An operator $A$ is said to be positive (denoted by $A\ge 0$) if $\left\langle Ax,x \right\rangle \ge 0$ for all $x\in \mathcal H$, and also an operator $A$ is said to be strictly positive (denoted by $A>0$) if $A$ is positive and invertible. The Gelfand map $f\left( t \right)\mapsto f\left( A \right)$ is an isometrically $*$-isomorphism between the ${{C}^{*}}$-algebra $C\left( \text{sp}\left( A \right) \right)$ of continuous functions on the spectrum $\text{sp}\left( A \right)$ of a self-adjoint operator $A$ and the ${{C}^{*}}$-algebra generated by ${{\mathbf 1}_{\mathcal H}}$ and $A$. If $f,g\in C\left( \text{sp}\left( A \right) \right)$, then $f\left( t \right)\ge g\left( t \right)$ ($t\in \text{sp}\left( A \right)$) implies that $f\left( A \right)\ge g\left( A \right)$.

On the other hand, when $A\in\mathcal{B}(\mathcal{H})$ is such that $\mathfrak{R}A >0$, then $A$ is said to be accretive.

When $\mathcal{H}$ is finite dimensional, we identify $\mathcal{B}(\mathcal{H})$ with the algebra $\mathcal{M}_n$ of all complex $n\times n$ matrices. Given a matrix monotone function $f:(0,\infty)\to (0,\infty)$ with $f(1)=1,$ and two accretive matrices $A,B$, there is a matrix mean associated with $f$, denoted by $\sigma_f$ or $\sigma$, defined by \cite{bedrani}
\begin{align}\label{eq_def_sigma_intro}
A\sigma B=A^{\frac{1}{2}}f\left(A^{-\frac{1}{2}}BA^{-\frac{1}{2}}\right)A^{\frac{1}{2}}.
\end{align}
In this formula, if $X$ is any matrix with no eigenvalue in $(-\infty,0]$, and if $f$ is analytic in a domain containing the eigenvalues of $X$, the quantity $f(X)$ is defined via the Dunford integral
$$f(X)=\frac{1}{2\pi i}\int_{\Gamma}f(z)(zI-X)^{-1}dz,$$ where $\Gamma$ is a closed simple curve in $\mathbb{C}$ surrounding the eigenvalues of $X$, and lying in the domain of analyticity of $f$. For such matrices, it has been recently shown in \cite{bedrani} that
\begin{equation}\label{eq_intro_f}
f(X)=\int_{0}^{1}I!_tX\;d\nu_f(t),
\end{equation}
 when $f:(0,\infty)\to (0,\infty)$ is  a matrix monotone function, for some probability measure $\nu_f$ on $[0,1].$ It is well known that when $A,B$ are accretive, then $A^{-\frac{1}{2}}BA^{-\frac{1}{2}}$ does not have any eigenvalue in $(0,\infty),$ \cite{drury}.\\
It should be noted that the identity in \eqref{eq_def_sigma_intro} is an extension of the same identity known for strictly positive matrices.

Using \eqref{eq_def_sigma_intro} and \eqref{eq_intro_f}, it can be easily seen that \cite{bedrani}
\begin{equation}\label{eq_sig_intro_int}
A\sigma_f B=\int_{0}^{1}A!_tB\;d\nu_f(t),
\end{equation}
where $\nu_f$ is the probability measure characterizing the matrix monotone function $f$ associated with $\sigma.$

When $A,B$ are accretive, it has been shown that \cite{bedrani}
\begin{equation}\label{eq_sigma_real_intro}
\Re(A\sigma_f B)\geq (\Re A)\;\sigma_f\;(\Re B)
\end{equation}
for any matrix mean $\sigma$ (defined via \eqref{eq_def_sigma_intro}). This last inequality is the extension of the corresponding inequalities known for the geometric and harmonic means of accretive matrices, \cite{drury,mfr}.\\
Reversing \eqref{eq_sigma_real_intro}, we have \cite{bedrani}
\begin{equation}\label{eq_sigma2_real_intro}
\Re(A\sigma_f B)\leq \sec^2(\alpha)(\Re A)\;\sigma_f\;(\Re B),
\end{equation}
where $A,B$ are sectorial matrices, with sectorial index $\alpha.$ Recall that a matrix $A$ is said to be sectorial with sectorial index $\alpha$, if $W(A)\subset S_{\alpha},$ where $W(A)$ is the numerical range of $A$ and $S_{\alpha}$ is the sector in the complex plane defined by
\begin{align*}
S_{\alpha}=\{z \in \mathbb{C}: \Re(z) > 0, |\Im(z)| \leq \tan(\alpha) \Re(z)  \}; 0\leq \alpha<\frac{\pi}{2}.
\end{align*}
One first main goal of this article is to find refinements of \eqref{eq_sigma_real_intro} and \eqref{eq_sigma2_real_intro}, via convex functions approach. Further inequalities for accretive matrices involving Hermite-Hadamard inequalities and concavity behavior will be presented too.

However, it turns out that this approach can be applied also to obtain some new versions of numerical radius and Tsallis relative entropy inequalities.\\
Recall that when $A\in\mathcal{B}(\mathcal{H})$, the numerical radius of $A$ is defined by
$$\omega(A)=
\sup_{\|x\|=1}\{|\left<Ax,x\right>|:x\in\mathbb{C},\|x\|=1\}.$$ 
This quantity has its applications in operator theory, where a considerable research has been conducted to find optimal bounds for $\omega(A)$; due to the difficulty in computing the exact value of this quantity. Among the most accurate upper bounds for $\omega(A)$ is the celebrated result of Kittaneh \cite[(8)]{nk} stating that
\begin{align}\label{eq_kitt_intro}
\omega \left( A \right)\le \frac{1}{2}\left\| \left| {{A}^{*}} \right|+\left| A \right| \right\|,
\end{align}
where $\|\cdot\|$ is the usual operator norm, and $|A|=(A^*A)^{\frac{1}{2}}$. We will present a new generalized form of this inequality.

The Tsallis relative operator entropy is defined for two strictly positive operators $A,B$ by 
$$
T_t(A|B):=A^{1/2}\ln_t\left(A^{-1/2}BA^{-1/2}\right)A^{1/2}=\frac{A\sharp_t B-A}{t}
$$
for $A,B \geq 0$ and $0<t \le 1$ \cite{ykf}. 
This converges to the relative operator entropy:
$$
\lim_{t\to 0}T_t(A|B)=A^{1/2}\log \left(A^{-1/2}BA^{-1/2}\right)A^{1/2}=:S(A|B).
$$

The key tool in our proofs is the following well known inequality for convex functions.

\begin{lemma}\label{10}
\cite{5} If $f:J\to \left( 0,\infty  \right)$ is a convex function, then
\begin{equation}\label{21}
f\left( \left( 1-t \right)a+tb \right)+2r\left( \frac{f\left( a \right)+f\left( b \right)}{2}-f\left( \frac{a+b}{2} \right) \right)\le \left( 1-t \right)f\left( a \right)+tf\left( b \right),
\end{equation}
and
\begin{equation}\label{22}
\left( 1-t \right)f\left( a \right)+tf\left( b \right)\le f\left( \left( 1-t \right)a+tb \right)+2R\left( \frac{f\left( a \right)+f\left( b \right)}{2}-f\left( \frac{a+b}{2} \right) \right),
\end{equation}
where $0\le t\le 1$, $r=\min \left\{ t,1-t \right\}$, and $R=\max \left\{ t,1-t \right\}$.
\end{lemma}
We refer the reader to \cite{sab_conv} for further related results to Lemma \ref{10}.

To state the following proposition, we need to remind the reader of interpolational means.\\
For a symmetric operator mean $\sigma $ (in the sense that $A\sigma  B=B\sigma A$), a parametrized operator mean ${{\sigma }_{t }}$  ($t \in \left[ 0,1 \right]$) is called an interpolational path for $\sigma $ (or Uhlmann's interpolation for $\sigma $) if it satisfies the following properties for the strictly positive operators $A,B$:
\begin{itemize}
	\item[(c1)] 	$A{{\sigma }_{0}}B=A$ (here we recall the convention ${{T}^{0}}=I$ for any positive operator $T$), $A{{\sigma }_{1}}B=B$, and $A{{\sigma }_{\frac{1}{2}}}B=A\sigma B$;
	\item[(c2)] $\left( A{{\sigma }_{\alpha }}B \right)\sigma \left( A{{\sigma }_{\beta }}B \right)=A{{\sigma }_{\frac{\alpha +\beta }{2}}}B$ for all $\alpha ,\beta \in \left[ 0,1 \right]$;
	\item[(c3)] the map $\alpha \in \left[ 0,1 \right]\mapsto A{{\sigma }_{\alpha }}B$  is norm continuous for each $A$ and $B$.
\end{itemize}
Examples of interpolational means are numerous, but the arithmetic, geometric and harmonic means are the most common. These are defined respectively as follows
$$A\nabla_t B=(1-t)A+t B, A\sharp_t B=A^{\frac{1}{2}}\left(A^{-\frac{1}{2}}BA^{-\frac{1}{2}}\right)^{t}A^{\frac{1}{2}}, A!_tB=((1-t)A^{-1}+t B^{-1})^{-1}.$$

The following proposition then applies.
\begin{proposition}\label{20}
Let $A,B\in \mathcal B\left( \mathcal H \right)$ be two strictly positive operators. Then 	
\[f\left( t \right)=\left\langle A{{\sigma }_{t}}Bx,x \right\rangle \]
is convex on $\left[ 0,1 \right]$. Further, if $\sigma\leq \sharp$, then $f$ is log-convex.
\end{proposition}
\begin{proof}
It is well-known that the arithmetic mean is the biggest one among symmetric means. This, together with \cite[Lemma 4]{7}, implies that
\[\begin{aligned}
   f\left( \frac{t+s}{2} \right)&=\left\langle A{{\sigma }_{\frac{t+s}{2}}}Bx,x \right\rangle  \\ 
 & =\left\langle \left( A{{\sigma }_{t}}B \right)\sigma \left( A{{\sigma }_{s}}B \right)x,x \right\rangle  \\ 
 & \le \left\langle A{{\sigma }_{t}}Bx,x \right\rangle \sigma \left\langle A{{\sigma }_{s}}Bx,x \right\rangle  \\ 
 & \le \frac{1}{2}\left( \left\langle A{{\sigma }_{t}}Bx,x \right\rangle +\left\langle A{{\sigma }_{s}}Bx,x \right\rangle  \right) \\ 
 & =\frac{1}{2}\left( f\left( t \right)+f\left( s \right) \right).  
\end{aligned}\]

On the other hand, if $\sigma\leq\sharp$, then the above computations show that
\begin{align*}
f\left(\frac{t+s}{2}\right)&\leq f(t)\sigma f(s)\leq f(t)\sharp f(s),
\end{align*}
which is equivalent to log-convexity of $f$.
\end{proof}
\begin{remark}
Letting $A=I$ in Proposition \ref{20} implies that the function
$f(t)=\left<B^tx,x\right>$ is log-convex on $[0,1],$ for any strictly positive operator $B$. Applying Lemma \ref{10} to this function implies the following multiplicative refinements of the celebrated McCarthy inequalities \cite{3}
$$\left<B^tx,x\right>\leq\left(\frac{\left<B^{\frac{1}{2}}x,x\right>}{\left<Bx,x\right>^{\frac{1}{2}}}\right)^{2r}\left<Bx,x\right>^{t}$$ and

$$\left<Bx,x\right>^t\leq\left(\frac{\left<Bx,x\right>^{\frac{1}{2}}}{\left<B^{\frac{1}{2}}x,x\right>}\right)^{2R}\left<B^{t}x,x\right>,$$
where $0\leq t\leq 1, r=\min\{t,1-t\}$ and $R=\max\{t,1-t\}.$
\end{remark}
\section{Main results}
As mentioned earlier in the introduction, we will be interested in presenting new and refined inequalities for accretive matrices, the numerical radius and the Tsallis relative operator entropy. These are the three parts of our main results.
\subsection{Accretive versions}
In this part of the paper, we present new inequalities for accretive matrices. We emphasize that this approach has never been tickled in the literature to treat accretive matrices, yet it was used for positive ones. In what follows, the notation $\sigma_t$ is implicitly understood to be the interpolational paths of the symmetric mean $\sigma$.
\begin{theorem}\label{4}
Let $A,B\in \mathcal{M}_n$ be two accretive matrices. If $0\le t\le 1$, then
	\[\mathfrak R\left( A{{\nabla }_{t}}B \right)\le \mathfrak R\left( A{{\sigma}_{t}}B \right)+2R\left( \mathfrak R\left( A\nabla B \right)-\left( \mathfrak RA \right)\sigma\left( \mathfrak RB \right) \right),\]
where $R=\max \left\{ t,1-t \right\}$.
\end{theorem}

\begin{proof}
We know that the function
	\[f\left( t \right)=\left\langle \left( \mathfrak RA \right){{\sigma}_{t}}\left( \mathfrak RB \right)x,x \right\rangle \]
is a convex function on $\left[ 0,1 \right]$, since $\mathfrak RA,\mathfrak RB>0$. On the other hand, it has been shown in \cite[Proposition 5.1]{bedrani} that
	\[\mathfrak R\left( A{{\sigma}_{t}}B \right)\ge \left( \mathfrak RA \right){{\sigma}_{t}}\left( \mathfrak RB \right).\]
These tools, together with Lemma \ref{10} imply that
	\[\begin{aligned}
  & \left\langle \left( \left( 1-t \right)\mathfrak{R}A+t\mathfrak{R}B \right)x,x \right\rangle  \\ 
 & \le \left\langle \left( \left( \mathfrak RA \right){{\sigma}_{t}}\left( \mathfrak RB \right)+2R\left( \frac{\left( \mathfrak RA \right)+\left( \mathfrak RB \right)}{2}-\left( \mathfrak RA \right)\sigma\left( \mathfrak RB \right) \right) \right)x,x \right\rangle  \\ 
 & \le \left\langle \left( \mathfrak R\left( A{{\sigma}_{t}}B \right)+2R\left( \frac{\left( \mathfrak RA \right)+\left( \mathfrak RB \right)}{2}-\left( \mathfrak RA \right)\sigma\left( \mathfrak RB \right) \right) \right)x,x \right\rangle   
\end{aligned}\]
Thus,
	\[\mathfrak R\left( A{{\nabla }_{t}}B \right)\le \mathfrak R\left( A{{\sigma}_{t}}B \right)+2R\left( \mathfrak R\left( A\nabla B \right)-\left( \mathfrak RA \right)\sigma\left( \mathfrak RB \right) \right).\]
\end{proof}

\begin{theorem}\label{6}
Let $A,B\in {{\mathcal M}_{n}}$ be accretive matrices such that $W\left( A \right),W\left( B \right)\subset {{S}_{\alpha }}$. Then, for $0\le t\le 1$,
\[\mathfrak R\left( A{{\sigma}_{t}}B \right)\le {{\sec }^{2}}\left( \alpha  \right)\left( \mathfrak R\left( A{{\nabla }_{t}}B \right)-2r\left( \mathfrak R\left( A\nabla B \right)-\left( \mathfrak RA \right){{\sigma}_{}}\left( \mathfrak RB \right) \right) \right),\]
where $r=\min \left\{ t,1-t \right\}$.
\end{theorem}
\begin{proof}
We know that for two sectorial matrices $A,B$ with sectorial index $\alpha$,
\[\mathfrak R\left( A{{\sigma}_{t}}B \right)\le {{\sec }^{2}}\left( \alpha  \right)\left( \left( \mathfrak RA \right){{\sigma}_{t}}\left( \mathfrak RB \right) \right)\]
holds \cite[Proposition 5.2]{bedrani}. Now, applying the same argument as in the proof of Theorem \ref{4}, we infer that
\[\begin{aligned}
  & \frac{1}{{{\sec }^{2}}\left( \alpha  \right)}\left\langle \mathfrak R\left( A{{\sigma}_{t}}B \right)x,x \right\rangle  \\ 
 & \le \left\langle \left( \mathfrak RA \right){{\sigma}_{t}}\left( \mathfrak RB \right)x,x \right\rangle  \\ 
 & \le \left\langle \left( \left( 1-t \right)\left( \mathfrak RA \right)+t\left( \mathfrak RB \right)-2r\left( \frac{\left( \mathfrak RA \right)+\left( \mathfrak RB \right)}{2}-\left( \mathfrak RA \right){{\sigma}_{}}\left( \mathfrak RB \right) \right) \right)x,x \right\rangle.
\end{aligned}\]
\end{proof}

\begin{remark}
If in Theorems \ref{4} and \ref{6}, $A$ and $B$ are positive matrices, we get
\[2r\left( A\nabla B-A\sigma B \right)\le A{{\nabla }_{t}}B-A{{\sigma }_{t}}B\le 2R\left( A\nabla B-A\sigma B \right),\]
which is a generalization of \cite[Corollary 3.1]{6}.
\end{remark}

For the harmonic mean, the following interesting refinement holds, refining \cite[Lemma 2.3]{mfr}.
\begin{theorem}\label{thm_har}
Let $A,B\in \mathcal{M}_n$ be two accretive matrices. If $0\le t\le 1$, then
\[\begin{aligned}
   \mathfrak R\left( A{{!}_{t}}B \right)&\ge {{\left( {{\left(\left( \mathfrak R A \right){{!}_{t}}\left( \mathfrak RB \right) \right)}^{-1}}-2r\left( {{\left( \left( \mathfrak RA \right)!\left(\mathfrak R B \right) \right)}^{-1}}-{{\left( \mathfrak R\left( A!B \right) \right)}^{-1}} \right) \right)}^{-1}} \\ 
 & \ge \left(\mathfrak  RA \right){{!}_{t}}\left(\mathfrak R B \right).
\end{aligned}\]
where $r=\min \left\{ t,1-t \right\}$.
\end{theorem}
\begin{proof}
If $f$ is operator convex, then
	\[g\left( t \right)=\left\langle f\left( \left( 1-t \right)A+tB \right)x,x \right\rangle \]
is convex on $\left[ 0,1 \right]$. This, together with Lemma \ref{10} imply that \cite{4}
	\[f\left( \left( 1-t \right)A+tB \right)\le \left( 1-t \right)f\left( A \right)+tf\left( B \right)-2r\left( \frac{f\left( A \right)+f\left( B \right)}{2}-f\left( \frac{A+B}{2} \right) \right).\]
On the other hand, 
\[f\left( \frac{A+B}{2} \right)\le \frac{f\left( A \right)+f\left( B \right)}{2},\]
which implies
	\[\begin{aligned}
   f\left( \left( 1-t \right)A+tB \right)&\le \left( 1-t \right)f\left( A \right)+tf\left( B \right)-2r\left( \frac{f\left( A \right)+f\left( B \right)}{2}-f\left( \frac{A+B}{2} \right) \right) \\ 
 & \le \left( 1-t \right)f\left( A \right)+tf\left( B \right). 
\end{aligned}\]
Since $f\left( T \right)={{\left( \mathfrak R{{T}^{-1}} \right)}^{-1}}$ is operator convex \cite[Theorem 2.2]{rm}, we get
{\small
\[\begin{aligned}
  & {{\left( \mathfrak R\left( {{\left( \left( 1-t \right)A+tB \right)}^{-1}} \right) \right)}^{-1}} \\ 
 & \le \left( 1-t \right){{\left( \mathfrak R{{A}^{-1}} \right)}^{-1}}+t{{\left( \mathfrak R{{B}^{-1}} \right)}^{-1}}-2r\left( \frac{{{\left( \mathfrak R{{A}^{-1}} \right)}^{-1}}+{{\left( \mathfrak R{{B}^{-1}} \right)}^{-1}}}{2}-{{\left( \mathfrak R\left( {{\left( \frac{A+B}{2} \right)}^{-1}} \right) \right)}^{-1}} \right) \\ 
 & \le \left( 1-t \right){{\left( \mathfrak R{{A}^{-1}} \right)}^{-1}}+t{{\left( R{{B}^{-1}} \right)}^{-1}}.
\end{aligned}\]
}By replacing $A$ and $B$ by ${{A}^{-1}}$ and ${{B}^{-1}}$, respectively, and then taking inverse, we infer that
{\small
\[\begin{aligned}
  & \mathfrak R\left( {{\left( \left( 1-t \right){{A}^{-1}}+t{{B}^{-1}} \right)}^{-1}} \right) \\ 
 & \ge {{\left( \left( 1-t \right){{\left( \mathfrak RA \right)}^{-1}}+t{{\left( \mathfrak RB \right)}^{-1}}-2r\left( \frac{{{\left( \mathfrak RA \right)}^{-1}}+{{\left( \mathfrak RB \right)}^{-1}}}{2}-{{\left( \mathfrak R\left( {{\left( \frac{{{A}^{-1}}+{{B}^{-1}}}{2} \right)}^{-1}} \right) \right)}^{-1}} \right) \right)}^{-1}} \\ 
 & \ge {{\left( \left( 1-t \right){{\left( \mathfrak RA \right)}^{-1}}+t{{\left( \mathfrak RB \right)}^{-1}} \right)}^{-1}} \\ 
\end{aligned}\]
}which is equivalent to
\[\begin{aligned}
   \mathfrak R\left( A{{!}_{t}}B \right)&\ge {{\left( {{\left(\left( \mathfrak R A \right){{!}_{t}}\left( \mathfrak RB \right) \right)}^{-1}}-2r\left( {{\left( \left( \mathfrak RA \right)!\left(\mathfrak R B \right) \right)}^{-1}}-{{\left( \mathfrak R\left( A!B \right) \right)}^{-1}} \right) \right)}^{-1}} \\ 
 & \ge \left(\mathfrak  RA \right){{!}_{t}}\left(\mathfrak R B \right).
\end{aligned}\]
This completes the proof.
\end{proof}
It is interesting that Theorem \ref{thm_har} implies the following refinement of \eqref{eq_sigma_real_intro}.
\begin{corollary}
Let $A,B$ be accretive matrices, and let $\sigma$ be a matrix mean, characterized by the matrix monotone function $f$. Then
\begin{align*}
\mathfrak{R}(A\sigma B)&\geq \int_{0}^{1}\left({{\left( {{\left(\left( \mathfrak R A \right){{!}_{t}}\left( \mathfrak RB \right) \right)}^{-1}}-2r\left( {{\left( \left( \mathfrak RA \right)!\left(\mathfrak R B \right) \right)}^{-1}}-{{\left( \mathfrak R\left( A!B \right) \right)}^{-1}} \right) \right)}^{-1}}\right)d\nu_f(t)\\
&\geq \mathfrak{R}A\sigma\mathfrak{R}B.
\end{align*}
\end{corollary}
\begin{proof}
This follows from Theorem \ref{thm_har}, upon integration, then using \eqref{eq_sig_intro_int}.
\end{proof}

In the next result, we present the accretive version of the first inequality in Lemma \ref{10}. 
\begin{theorem}
Let $f:\left( 0,\infty  \right)\to \left( 0,\infty  \right)$ be a matrix concave with $f\left( 1 \right)=1$, and let $A,B$ be two accretive matrices with $W\left( A \right),W\left( B \right)\subset {{S}_{\alpha }}$ for some $0\le \alpha <{\pi }/{2}\;$. Then for any $0\le t\le 1$
\[\mathfrak R\left( f\left( A \right){{\nabla }_{t}}f\left( B \right) \right)+2r{{\sec }^{2}}\left( \alpha  \right)\left( f\left( \mathfrak R\left( A\nabla B \right) \right)-f\left( \mathfrak RA \right)\nabla f\left( \mathfrak RB \right) \right)\le {{\sec }^{2}}\left( \alpha  \right)\mathfrak Rf\left( A{{\nabla }_{t}}B \right),\]
where $r=\min \left\{ t,1-t \right\}$. 
\end{theorem}
\begin{proof}
If $f:\left( 0,\infty  \right)\to \left( 0,\infty  \right)$ is  matrix concave with $f\left( 1 \right)=1$ and $A$ is  accretive, then \cite[Proposition 7.1]{bedrani}
	\[f\left( \mathfrak RA \right)\le \mathfrak Rf\left( A \right).\]
Furthermore, if $W\left( A \right)\subset {{S}_{\alpha }}$ for some $0\le \alpha <{\pi }/{2}\;$, then \cite[Proposition 7.2]{bedrani}
	\[\mathfrak Rf\left( A \right)\le {{\sec }^{2}}\left( \alpha  \right)f\left( \mathfrak RA \right).\]
Now, assume that $0\le t\le {1}/{2}\;$. Then using \cite[Propositions 7.1, 7.2]{bedrani}, we have
\[\begin{aligned}
  & \mathfrak Rf\left( A{{\nabla }_{t}}B \right) \\ 
 & =\mathfrak Rf\left( \left( 1-t \right)A+tB \right) \\ 
 & \ge f\left( \left( 1-t \right)\mathfrak RA+t\mathfrak RB \right) \\ 
 & =f\left( \left( 1-2t \right)\mathfrak RA+2t\frac{\mathfrak RA+\mathfrak RB}{2} \right) \\ 
 & \ge \left( 1-2t \right)f\left( \mathfrak RA \right)+2tf\left( \frac{\mathfrak RA+\mathfrak RB}{2} \right) \\ 
 & =\left( 1-t \right)f\left( \mathfrak RA \right)+tf\left( \mathfrak RB \right)+2r\left( f\left( \frac{\mathfrak RA+\mathfrak RB}{2} \right)-\frac{f\left( \mathfrak RA \right)+f\left( \mathfrak RB \right)}{2} \right) \\ 
 & \ge \frac{1}{{{\sec }^{2}}\left( \alpha  \right)}\left( \left( 1-t \right)\mathfrak Rf\left( A \right)+t\mathfrak Rf\left( B \right) \right)+2r\left( f\left( \frac{\mathfrak RA+\mathfrak RB}{2} \right)-\frac{f\left( \mathfrak RA \right)+f\left( \mathfrak RB \right)}{2} \right) \\ 
 & =\frac{1}{{{\sec }^{2}}\left( \alpha  \right)}\mathfrak R\left( f\left( A \right){{\nabla }_{t}}f\left( B \right) \right)+2r\left( f\left( \mathfrak R\left( A\nabla B \right) \right)-f\left( \mathfrak RA \right)\nabla f\left( \mathfrak RB \right) \right), 
\end{aligned}\]
where $r=\min \left\{ t,1-t \right\}$. Consequently
\[\mathfrak R\left( f\left( A \right){{\nabla }_{t}}f\left( B \right) \right)+2r{{\sec }^{2}}\left( \alpha  \right)\left( f\left( \mathfrak R\left( A\nabla B \right) \right)-f\left( \mathfrak RA \right)\nabla f\left( \mathfrak RB \right) \right)\le {{\sec }^{2}}\left( \alpha  \right)\mathfrak Rf\left( A{{\nabla }_{t}}B \right).\]
The same inequality holds when ${1}/{2}\;\le t\le 1$. This completes the proof.
 \end{proof}

 \begin{theorem}
(Hermite-Hadamard inequality for accretive matrices) Let $f:\left( 0,\infty  \right)\to \left( 0,\infty  \right)$ be a matrix concave with $f\left( 1 \right)=1$, and let $A,B$ be two accretive matrices with $W\left( A \right),W\left( B \right)\subset {{S}_{\alpha }}$ for some $0\le \alpha <{\pi }/{2}\;$. Then 
\[\mathfrak R\left( \frac{f\left( A \right)+f\left( B \right)}{2} \right)\le {{\sec }^{2}}\left( \alpha  \right)\int\limits_{0}^{1}{\mathfrak Rf\left( \left( 1-t \right)A+tB \right)dt}\le {{\sec }^{4}}\left( \alpha  \right)\mathfrak Rf\left( \frac{A+B}{2} \right).\]
 \end{theorem}
 \begin{proof}
 Observe that
\[\begin{aligned}
   \mathfrak Rf\left( \frac{A+B}{2} \right)&=\mathfrak Rf\left( \frac{\left( 1-t \right)A+tB+\left( 1-t \right)B+tA}{2} \right) \\ 
 & \ge \frac{1}{{{\sec }^{2}}\left( \alpha  \right)}\frac{\mathfrak Rf\left( \left( 1-t \right)A+tB \right)+\mathfrak Rf\left( \left( 1-t \right)B+tA \right)}{2} \\ 
 & \ge \frac{1}{{{\sec }^{4}}\left( \alpha  \right)}\frac{\mathfrak Rf\left( A \right)+\mathfrak Rf\left( B \right)}{2},
\end{aligned}\]
where we have used  \cite[Theorem 7.2]{bedrani} twice to obtain the first and second inequalities.
Equivalently,
\[\begin{aligned}
   \mathfrak R\left( \frac{f\left( A \right)+f\left( B \right)}{2} \right)&\le {{\sec }^{2}}\left( \alpha  \right)\mathfrak R\left( \frac{f\left( \left( 1-t \right)A+tB \right)+f\left( \left( 1-t \right)B+tA \right)}{2} \right) \\ 
 & \le {{\sec }^{4}}\left( \alpha  \right)\mathfrak Rf\left( \frac{A+B}{2} \right).  
\end{aligned}\]
By taking integral over $0\le t\le 1$, we reach the desired result.
 \end{proof}

\subsection{Numerical Radius Inequalities} It has been shown in \cite[Theorem 2]{ne} that if $A\in \mathcal B\left( \mathcal H \right)$, then
\begin{equation}\label{5}
{{\omega }^{2p}}\left( A \right)\le \left\| \left( 1-t \right){{\left| {{A}^{*}} \right|}^{2p}}+t{{\left| A \right|}^{2p}} \right\|,
\end{equation}
for any $0\le t\le 1$ and $p\ge 1$. The next result improves the inequality \eqref{5}.
\begin{theorem}\label{23}
Let $A\in \mathcal B\left( \mathcal H \right)$ and let $0\le t\le 1$. Then for any $p\ge 1$,
\[{{\omega }^{2p}}\left( A \right)\le \left\| \left( 1-t \right){{\left| {{A}^{*}} \right|}^{2p}}+t{{\left| A \right|}^{2p}}-2r\left( \frac{{{\left| A \right|}^{2p}}+{{\left| {{A}^{*}} \right|}^{2p}}}{2}-{{\left( \frac{{{\left| A \right|}^{p}}+{{\left| {{A}^{*}} \right|}^{p}}}{2} \right)}^{2}} \right) \right\|,\]
where $r=\min \left\{ t,1-t \right\}$ and $R=\max \left\{ t,1-t \right\}$.
\end{theorem}
\begin{proof}
Since the function
\[f\left( t \right)=\left\langle {{\left| A \right|}^{2pt}}x,x \right\rangle \left\langle {{\left| {{A}^{*}} \right|}^{2p\left( 1-t \right)}}x,x \right\rangle \]
is log-convex on $\left( 0,\infty  \right)$, so is convex. Hence, by Lemma \ref{10}, we obtain	
\[\begin{aligned}
  & \left\langle {{\left| A \right|}^{2pt}}x,x \right\rangle \left\langle {{\left| {{A}^{*}} \right|}^{2p\left( 1-t \right)}}x,x \right\rangle +2r\left( \frac{\left\langle {{\left| A \right|}^{2p}}x,x \right\rangle +\left\langle {{\left| {{A}^{*}} \right|}^{2p}}x,x \right\rangle }{2}-\left\langle {{\left| A \right|}^{p}}x,x \right\rangle \left\langle {{\left| {{A}^{*}} \right|}^{p}}x,x \right\rangle  \right) \\ 
 & \le \left( 1-t \right)\left\langle {{\left| {{A}^{*}} \right|}^{2p}}x,x \right\rangle +t\left\langle {{\left| A \right|}^{2p}}x,x \right\rangle ,
\end{aligned}\]
where $r=\min \left\{ t,1-t \right\}$, $R=\max \left\{ t,1-t \right\}$, and $0\le t\le 1$. By applying the arithmetic-geometric mean inequality, we infer that
\[\begin{aligned}
  & \left\langle {{\left| A \right|}^{2pt}}x,x \right\rangle \left\langle {{\left| {{A}^{*}} \right|}^{2p\left( 1-t \right)}}x,x \right\rangle  \\ 
 & \le \left\langle \left( \left( 1-t \right){{\left| {{A}^{*}} \right|}^{2p}}+t{{\left| A \right|}^{2p}} \right)x,x \right\rangle +2r\left( \left\langle {{\left| A \right|}^{p}}x,x \right\rangle \left\langle {{\left| {{A}^{*}} \right|}^{p}}x,x \right\rangle -\left\langle \frac{{{\left| A \right|}^{2p}}+{{\left| {{A}^{*}} \right|}^{2p}}}{2}x,x \right\rangle  \right) \\ 
 & \le \left\langle \left( \left( 1-t \right){{\left| {{A}^{*}} \right|}^{2p}}+t{{\left| A \right|}^{2p}} \right)x,x \right\rangle +2r\left( {{\left( \frac{\left\langle {{\left| A \right|}^{p}}x,x \right\rangle +\left\langle {{\left| {{A}^{*}} \right|}^{p}}x,x \right\rangle }{2} \right)}^{2}}-\left\langle \frac{{{\left| A \right|}^{2p}}+{{\left| {{A}^{*}} \right|}^{2p}}}{2}x,x \right\rangle  \right) \\ 
 & =\left\langle \left( \left( 1-t \right){{\left| {{A}^{*}} \right|}^{2p}}+t{{\left| A \right|}^{2p}} \right)x,x \right\rangle +2r\left( {{\left\langle \left( \frac{{{\left| A \right|}^{p}}+{{\left| {{A}^{*}} \right|}^{p}}}{2} \right)x,x \right\rangle }^{2}}-\left\langle \frac{{{\left| A \right|}^{2p}}+{{\left| {{A}^{*}} \right|}^{2p}}}{2}x,x \right\rangle  \right) \\ 
 & \le \left\langle \left( \left( 1-t \right){{\left| {{A}^{*}} \right|}^{2p}}+t{{\left| A \right|}^{2p}} \right)x,x \right\rangle +2r\left( \left\langle {{\left( \frac{{{\left| A \right|}^{p}}+{{\left| {{A}^{*}} \right|}^{p}}}{2} \right)}^{2}}x,x \right\rangle -\left\langle \frac{{{\left| A \right|}^{2p}}+{{\left| {{A}^{*}} \right|}^{2p}}}{2}x,x \right\rangle  \right),
\end{aligned}\]
where in the last inequality we used the H\"older-McCarthy inequality. Thus,
\[\begin{aligned}
  & \left\langle {{\left| A \right|}^{2pt}}x,x \right\rangle \left\langle {{\left| {{A}^{*}} \right|}^{2p\left( 1-t \right)}}x,x \right\rangle  \\ 
 & \le \left\langle \left( \left( 1-t \right){{\left| {{A}^{*}} \right|}^{2p}}+t{{\left| A \right|}^{2p}} \right)x,x \right\rangle +2r\left( \left\langle {{\left( \frac{{{\left| A \right|}^{p}}+{{\left| {{A}^{*}} \right|}^{p}}}{2} \right)}^{2}}x,x \right\rangle -\left\langle \frac{{{\left| A \right|}^{2p}}+{{\left| {{A}^{*}} \right|}^{2p}}}{2}x,x \right\rangle  \right). 
\end{aligned}\]
Since 
	\[\begin{aligned}
   {{\left| \left\langle Ax,x \right\rangle  \right|}^{2p}}&\le {{\left( \left\langle {{\left| A \right|}^{2t}}x,x \right\rangle \left\langle {{\left| {{A}^{*}} \right|}^{2\left( 1-t \right)}}x,x \right\rangle  \right)}^{p}} \\ 
 & \le \left\langle {{\left| A \right|}^{2pt}}x,x \right\rangle \left\langle {{\left| {{A}^{*}} \right|}^{2p\left( 1-t \right)}}x,x \right\rangle,   
\end{aligned}\]
we get
	\[\begin{aligned}
  & {{\left| \left\langle Ax,x \right\rangle  \right|}^{2p}} \\ 
 & \le \left\langle \left( \left( 1-t \right){{\left| {{A}^{*}} \right|}^{2p}}+t{{\left| A \right|}^{2p}} \right)x,x \right\rangle +2r\left( \left\langle {{\left( \frac{{{\left| A \right|}^{p}}+{{\left| {{A}^{*}} \right|}^{p}}}{2} \right)}^{2}}x,x \right\rangle -\left\langle \frac{{{\left| A \right|}^{2p}}+{{\left| {{A}^{*}} \right|}^{2p}}}{2}x,x \right\rangle  \right). \\ 
\end{aligned}\]
By taking supremum over all unit vector $x\in \mathcal H$ we get the desired result.
\end{proof}
\begin{remark}
Theorem \ref{23} extends the celebrated inequality \cite[(8)]{nk}
	\[\omega \left( A \right)\le \frac{1}{2}\left\| \left| {{A}^{*}} \right|+\left| A \right| \right\|,\]
by setting $t={1}/{2}\;$.
\end{remark}

\subsection{Tsallis Relative Operator Entropy}
We consider $t$-logarithmic function $\ln_tx:=\dfrac{x^t-1}{t}$ defined for $x>0$ and $t>0$.
\begin{lemma}
The $t$-logarithmic function $\ln_tx$ is convex in $t$ if $x \geq 1$, and concave in $t$ if $0<x \le 1$.
\end{lemma}
\begin{proof}
We set the function $f(y):=y(\log y)^2-2y\log y+2y-2$ for $y>0$.
Since $f(1)=0$ and $f'(y)=(\log y)^2 \geq 0$, we have
$f(y)\le 0$ for $0<y\le 1$ and $f(y)\ge 0$ for $y \ge 1$.

By simple calculations, we have
$$
\frac{d^2}{dt^2}\left(\ln_t x\right)=\frac{1}{t^3}\left\{x^t(\log x^t)^2-2x^t\log x^t+2x^t-2\right\}.
$$
By putting $y:=x^t$ in the above, we thus get $\dfrac{d^2}{dt^2}\left(\ln_t x\right)\ge 0$ if $x \geq 1$ and $\dfrac{d^2}{dt^2}\left(\ln_t x\right)\le  0$ if $0<x \le 1$.
\end{proof}

\begin{proposition}
If $0\le A \le B$, then we have
$$
T_{(1-t)a+tb}(A|B)\le (1-t) T_a(A|B) +tT_b(A|B)
$$
for all $0<a,b\le 1$ and $0\le t \le 1$. If $0\le B \le A$, then the reverse inequality above holds.
\end{proposition}
\begin{proof}
Since $\ln_t x$ is convex in $t$ for $x \geq 1$, we have
$$
\ln_{(1-t)a+tb}x \le (1-t)\ln_ax+t\ln_bx
$$
for all $0<a,b\le 1$, $0\le t \le 1$ and $x \geq 1$.
Putting $x:=A^{-1/2}BA^{-1/2}$ and then multiplying $A^{1/2}$ to the both sides in the inequality, we get the desired result. For the case $0<x \le 1$, we also obtain the result similarly.
\end{proof}

\begin{theorem}
For $0<t\le 1$ and $A,B\ge 0$, we have
\begin{eqnarray*}
2r\left(\frac{B-A+S(A|B)}{2}-2\left(A\sharp B-A\right)\right)
&\le& (1-t)S(A|B)+t(B-A)- T_t(A|B)\\
&\le& 2R\left(\frac{B-A+S(A|B)}{2}-2\left(A\sharp B-A\right)\right),
\end{eqnarray*}
where $r:=\min \{t,1-t\}$ and $R:=\max\{t,1-t\}$.
\end{theorem}
\begin{proof}
We consider $f(t):=\ln_t x$ for $0<t\le 1$ and $x>0$.
Since $f(0)=\lim\limits_{t\to 0}\dfrac{x^t-1}{t}=\log x$, $f(1/2)=2(\sqrt{x}-1)$ and $f(1)=x-1$, by Lemma \ref{10}, we have
{\small
$$
2r\left(\frac{x-1+\log x}{2}-2\left(\sqrt{x}-1\right)\right)\le (1-t)\log x+t(x-1)-\ln_tx \le 2R\left(\frac{x-1+\log x}{2}-2\left(\sqrt{x}-1\right)\right).
$$
}
From Kubo-Ando theory, we obtain the desired results.
\end{proof}

Note that we consequently obtain a natural result $T_{1/2}(A|B)=2\left(A\sharp B -A\right)$ if we take $t=1/2$ in the above theorem.

\begin{theorem}
Let $A,B\in \mathcal B\left( \mathcal H \right)$ be two strictly positive operators. If $t\le s$, then
\[{{T}_{t}}\left( A|B \right)\le {{T}_{s}}\left( A|B \right).\]
\end{theorem}
\begin{proof}
We know that if $g$ is convex function, then the function
\[h\left( t \right)=\frac{g\left( t \right)-g\left( a \right)}{t-a}\]
is increasing. This means that the function 
\[f\left( t \right)=\left\langle \left( \frac{A{{\sharp}_{t}}B-A}{t} \right)x,x \right\rangle =\left\langle {{T}_{t}}\left( A|B \right)x,x \right\rangle \]
is increasing, thanks to Proposition \ref{20}.
\end{proof}

\section*{Acknowledgment} 
The author (S.F.) was partially supported by JSPS KAKENHI Grant Number 21K03341.

\vskip 0.6 true cm

\tiny(H.R. Moradi) Young Researchers and Elite Club, Mashhad Branch, Islamic Azad University, Mashhad, Iran

{\it E-mail address:} hrmoradi@mshdiau.ac.ir

\vskip 0.4 true cm

(S. Furuichi) Department of Information Science, College of Humanities and Sciences, Nihon University, 3-25-40, Sakurajyousui, Setagaya-ku, Tokyo, 156-8550, Japan.

{\it E-mail address:} furuichi@chs.nihon-u.ac.jp

\vskip 0.4 true cm
{\tiny (M. Sababheh) Department of Basic Sciences, Princess Sumaya University for Technology, Amman 11941,
	Jordan. 
	
	\textit{E-mail address:} sababheh@yahoo.com; sababheh@psut.edu.jo}

\end{document}